\documentclass{article}
\usepackage{amssymb}
\usepackage{amsfonts}
\usepackage{amsmath}
\usepackage[height=190mm,width=130mm]{geometry}
\newcommand{\func}[1]{\operatorname{#1}}
\newtheorem{theorem}{Theorem}[section]

\newtheorem{conclusion}{Conclusion}[section]

\newtheorem{corollary}{Corollary}[section]

\newtheorem{definition}{Definition}[section]

\newtheorem{lemma}{Lemma}[section]

\newtheorem{proposition}{Proposition}[section]

\newenvironment{proof}[1][Proof]{\noindent\textbf{#1.} }{\ \rule{0.5em}{0.5em}}
\numberwithin{equation}{section}
\def\({\left ( }
\def\){\right )}
\def\<{\left < }
\def\>{\right >}

\begin{document}

\title{\textbf{On the $\mathit{g}$-Circulant
Matrix involving the Generalized $\mathit{k}$-Horadam Numbers}}
\author {$^{(1)}$\textbf{Nazmiye Yilmaz}\,\thanks{e mail: nzyilmaz@selcuk.edu.tr,
 yyazlik@nevsehir.edu.tr, ntaskara@selcuk.edu.tr.},  $^{(2)}$\textbf{Yasin Yazlik },
 $^{(1)}$\textbf{Necati Taskara}} 

\date{$^{(1)}$Department of Mathematics, Faculty of Science,\\
Selcuk University, Campus, 42075, Konya, Turkey\\
$^{(2)}$Department of Mathematics, Faculty of Science and Art,\\
Nevsehir Haci Bektasi Veli University, 50300, Nevsehir, Turkey}
\maketitle

\begin{abstract}
\par
In this study, we present a new generalization of circulant matrices for the
generalized $k$-Horadam numbers, by considering the $g$-circulant matrix $C_{n,g}(H)=g
-circ(H_{k,1},H_{k,2},\ldots ,H_{k,n})$. Also, we calculate the spectral norm, determinant and inverse of $C_{n,g}(H)$ in such matrices having the elements of all second order sequences.
\end{abstract}

\section{Introduction and Preliminaries}

Many generalizations of the Fibonacci sequence have been introduced and
studied \cite{FalconPlaza,Horadam,Koshy,Usluveark1,YazlikTaskara1}. Here we use the generalized $k$-Horadam numbers as follows.

Let $k$ be any positive real number and $f\left( k\right) ,\ g\left(
k\right) $ are scaler-value polynomials and $f^{2}(k)+4g(k)>0.$\ For $n\geq
0,\ $the generalized $k$-Horadam sequence $\left\{ H_{k,n}\right\} _{n\in 
\mathbb{N}}$ is defined by 
\begin{equation}
H_{k,n+2}=f(k)H_{k,n+1}+g(k)H_{k,n}, H_{k,0}=a,H_{k,1}=b.\label{1.1}
\end{equation}
where $a,$ $b\in\mathbb{R}$ \cite{YazlikTaskara1}.
Obviously, if we choose suitable values on $f(k),g(k),a$ and $b$ in (\ref{1.1})
then this sequence reduces to the special all second order sequences
in the literature. For example, by taking $f(k)=g(k)=1$, $a=0$ and $b=1$,
then it is obtained the well known Fibonacci sequence.

Let $r_{1}$ and $r_{2}$ be the roots of the characteristic equation $x^{2}-f(k)x-g(k)=0$ of (\ref{1.1}). Then the Binet formula of this sequence $\left\{ H_{k,n}\right\} _{n\in\mathbb{N}}$ have the form 
\begin{equation}
H_{k,n}=\dfrac{Xr_{1}^{n}-Yr_{2}^{n}}{r_{1}-r_{2}},  \label{1.2}
\end{equation}
where $X=b-ar_{2}$ and$\ Y=b-ar_{1}$. Also, the summation of this sequence is
given by
\begin{equation}
\sum_{i=1}^{n}H_{k,i}=\frac{H_{k,n+1}+g\left( k\right)
H_{k,n}-H_{k,1}-g\left( k\right) H_{k,0}}{f\left( k\right) +g\left( k\right)
-1},  \label{1.3}
\end{equation}
where $f\left( k\right) +g\left( k\right) -1\neq 0$.
\par
The $g$-circulant matrices have been one of the most important and active
research field of applied mathematic and computation mathematic
increasingly. There are lots of examples from statistical and information
theory illustrate applications of the $g$-circulant matrices, which
emphasize how the asymptotic eigenvalue distribution theorem allows one to
evaluate results for processes (see \cite{Chouveark,Grayveark,Horn} and therein). In the last years, there have been several papers on circulant matrices \cite{Alptekinveark}-\cite{Chouveark},\cite{Grayveark,Good},\cite{Ipek,Jiangveark,Kocer},\cite{Radicic}-\cite{TugluKizilates},\cite{YazlikTaskara2}-\cite{Zhouveark}. For
instance, Alptekin, Mansour and Tuglu, \cite{Alptekinveark}, obtained the spectral
norm and eigenvalues of circulant matrices with Horadam's numbers. Also,
they defined the semicirculant matrix with these numbers and give Euclidean
norm of this matrix. The authors in \cite{Jiangveark} defined $g$-circulant matrices with $k$-Fibonacci and $k$-Lucas numbers and computed the determinant and the inverse of these matrices. In \cite{Kocer}, it was studied the norms, eigenvalues and determinants of some matrices related to different numbers. In \cite{Shenveark}, authors defined the $n\times n$ circulant matrices $A=$ $\left[ a_{ij}\right]$ and $B=\left[ b_{ij}\right]$,
where $a_{ij}\equiv F_{\left(\func{mod}\left( j-i,n\right)\right)}$ and $
b_{ij}\equiv L_{\left(\func{mod}\left(j-i,n\right)\right)}$. Also, the
inverses of matrices $A$ and $B$ were derived. In addition, Solak \cite{Solak} defined  the $n\times n$ circulant matrices $A=\left[ a_{ij}\right]$ and $B=\left[ b_{ij}\right]$, where $a_{ij}\equiv F_{\left(\func{mod}\left( j-i,n\right) \right)}$ and $b_{ij}\equiv L_{\left(\func{mod}\left( j-i,n\right) \right) }$. He
investigated the upper and lower bounds of the matrices $A$ and $B$. Additionally, Yazlik and Taskara \cite{YazlikTaskara2,YazlikTaskara3} defined circulant matrix $C_{n}(H)$ whose entries are the generalized $k$-Horadam numbers and computed the spectral norm, eigenvalues, determinant and the inverse of this matrix. That is, authors gave the determinant and inverse of matrix $C_{n}(H)$ as follows:
\begin{equation}
\det C_{n}\left( H\right)
=H_{k,1}N^{n-1}+H_{k,1}M^{n-2}\sum\limits_{i=1}^{n-1}\left( -\frac{
H_{k,2}H_{k,i+1}}{H_{k,1}}+H_{k,i+2}\right) \left( \frac{N}{M}\right) ^{i-1},
\label{1.4}
\end{equation}
and for $n>2,$
\begin{eqnarray}
C_{n}^{-1}\left( H\right) &=& circ(\frac{1+f(k)S_{n}^{(n-2)}+g(k)S_{n}^{(n-3)}
}{h_{n}}\frac{g(k)S_{n}^{(n-2)}-\frac{H_{k,2}}{H_{k,1}}}{h_{n}}
,  \label{1.5} \\
&&{-}\frac{S_{n}^{(1)}}{h_{n}},{ -}\frac{
S_{n}^{(2)}-f(k)S_{n}^{(1)}}{h_{n}},{ -}\frac{
S_{n}^{(3)}-f(k)S_{n}^{(2)}-g(k)S_{n}^{(1)}}{h_{n}}  \nonumber \\
&&{ ,\ldots ,-}\frac{S_{n}^{(n-2)}-f(k)S_{n}^{(n-3)}-g(k)S_{n}^{(n-4)}}{
h_{n}}){ ,}  \nonumber
\end{eqnarray}
where 
\begin{eqnarray*}
S_{n}^{(j)} &=& \sum\limits_{i=1}^{j}\frac{\left( H_{k,j+3-i}-
\frac{H_{k,2}H_{k,j+2-i}}{H_{k,1}}\right) M^{i-1}}{N^{i}}, (j=1,2,\ldots
,n-2), 
\cr h_{n} &=& -\frac{H_{k,2}H_{k,n}}{H_{k,1}}+H_{k,1}+\sum
\limits_{i=1}^{n-1}\left( -\frac{H_{k,2}H_{k,i+1}}{H_{k,1}}+H_{k,i+2}\right)
\left( \frac{M}{N}\right) ^{n-\left( i+1\right) },
\cr M &=& g(k)(H_{k,n}-H_{k,0}) \, \text{and} \, N=H_{k,1}-H_{k,n+1}.
\end{eqnarray*}
\par
Now we give some preliminaries related our study. $A$, $g$-circulant matrix
is an $n\times n$ complex matrix with the following form
\begin{equation*}
A=\left( 
\begin{array}{cccc}
a_{0} & a_{1} & \cdots & a_{n-1} \\ 
a_{n-g} & a_{n-g+1} & \cdots & a_{n-g-1} \\ 
a_{n-2g} & a_{n-2g+1} & \cdots & a_{n-2g-1} \\ 
\vdots & \vdots & \ddots & \vdots \\ 
a_{g} & a_{g+1} & \cdots & a_{g-1}
\end{array}
\right) ,
\end{equation*}
where $g$ is nonnegative integer and each of the subscripts is understood to
be reduced modulo $n$. The first row of $A$ is $(a_{0},a_{1},\ldots ,a_{n-1})$ and its $\left(j+1\right)-$th row is obtained by giving $j-$th row a right circular shift
by g positions. Note that, $g=1$ or $g=n+1$ yields the classical circulant
matrix \cite{Stallingsveark}.
\par
From \cite{Horn}, we further remind that, for a matrix $A=\left[ a_{i,j}\right]
\in M_{m,n}(\mathbb{C})$, the spectral norm of $A$ is given by 
\begin{equation*}
\left\Vert A\right\Vert _{2}=\sqrt{\max\limits_{1\leq i\leq n}\lambda
_{i}\left( A^{\ast }A\right) } 
\end{equation*}
where $A^{\ast }$ is the conjugate transpose of matrix $A$.
\begin{lemma}\label{lem1}
\cite{Horn} Let $A$ be an $n\times n$ matrix with eigenvalues $\lambda
_{1},\lambda _{2},\ldots ,\lambda _{n}$. If $A$ is a normal matrix, then 
\begin{equation*}
\left\Vert A\right\Vert _{2}=\max\limits_{1\leq i\leq n}\left\vert \lambda
_{i}\right\vert\,.
\end{equation*}
\end{lemma}
\begin{lemma}\label{lem2}
\cite{Stallingsveark} An $n\times n$ matrix $Q_{g}$ is unitary if and only if
\begin{equation*}
\left( n,g\right) =1,
\end{equation*}
where $Q_{g}$ is a $g$-circulant matrix with the first row $e^{\ast }=\left(
1,0,\ldots 0\right)$.
\end{lemma}
\begin{lemma}\label{lem3}
\cite{Stallingsveark} $A$ is $g$-circulant matrix with the first row $(a_{0},a_{1},\ldots ,a_{n-1})$ if and only if
\begin{equation*}
A=Q_{g}C,
\end{equation*}
where $C$ is circulant matrix, that is, $C=circ(a_{0},a_{1},\ldots
,a_{n-1}). $
\end{lemma}
\par
In the light of all these above material (depicted as separate paragraphs), the main goal of this paper is to investigate \textit{the properties of $g$-circulant matrix with $k$-Horadam numbers}. To do that we consider $g$-circulant matrix $C_{n,g}(H)=g$-$circ(H_{k,1},H_{k,2},\ldots ,H_{k,n}),$ where $H_{k,n}$ is the generalized $k$-Horadam numbers. Firstly, we obtain the values of the spectral norm and determinant of this matrix can be expressed with only the generalized $k$-Horadam numbers. Also we formulate the inverse of $g$-circulant matrix $C_{n,g}(H)$. In fact, the results in here are the most general statements to obtain the spectral norms, determinants and inverses in such matrices having the elements of all second order sequences.

\section{Main Results}

\begin{definition}
An $(n\times n)$ $g$-circulant matrix with generalized $k$-Horadam numbers
entries is defined by 
\begin{equation}
C_{n,g}\left( H\right) =\left( 
\begin{array}{ccccc}
H_{k,1} & H_{k,2} & H_{k,3} & ... & H_{k,n} \\ 
H_{k,n-g+1} & H_{k,n-g+2} & H_{k,n-g+3} & ... & H_{k,n-g} \\ 
H_{k,n-2g+1} & H_{k,n-2g+2} & H_{k,n-2g+3} & ... & H_{k,n-2g} \\ 
\vdots & \vdots & \vdots & \ddots & \vdots \\ 
H_{k,g+1} & H_{k,g+2} & H_{k,g+3} & ... & H_{k,g}%
\end{array}
\right) ,  \label{2.1}
\end{equation}
where $g$ is nonnegative integer.
\end{definition}
\par
The following theorem gives us the values of the determinant of this matrix
can be expressed by utilizing the generalized $k$-Horadam numbers.

\begin{theorem}\label{teo1}
Let $C_{n,g}\left( H\right) =g$-$circ\left( H_{k,1},H_{k,2},\ldots
H_{k,n}\right) $ be circulant matrix as in (\ref{2.1}). Then we have
\begin{equation*}
\left\Vert C_{n,g}\left( H\right) \right\Vert _{2}=\frac{H_{k,n+1}+g\left(
k\right) H_{k,n}-H_{k,1}-g\left( k\right) H_{k,0}}{f\left( k\right) +g\left(
k\right) -1}\,, 
\end{equation*}
where $f\left( k\right) +g\left( k\right) -1\neq 0$.
\end{theorem}

\begin{proof}
We express that $g$-circulant matrix is normal and irreducible (see \cite{Zhouveark}).
So, the spectral norm of $C_{n,g}\left( H\right) $ is given by the spectral
radius of $C_{n,g}\left( H\right) $. Also since $C_{n,g}\left( H\right)$ is
irreducible and entrywise nonnegative, its spectral radius is the same as
its Perron value. Let $u$ denote an all ones vector of order $n$.Then $
C_{n,g}\left( H\right) u=\left( \sum_{i=1}^{n}H_{k,i}\right) u$. As $
\sum_{i=1}^{n}H_{k,i}$ is an eigenvalue of $C_{n,g}\left( H\right)$
associated with a positive eigenvector, it is necessarily the Perron value
of $C_{n,g}\left( H\right)$. Hence, from the Equation (\ref{1.3}), we
conclude that
\begin{equation*}
\left\Vert C_{n,g}\left( H\right) \right\Vert _{2}=\frac{H_{k,n+1}+g\left(
k\right) H_{k,n}-H_{k,1}-g\left( k\right) H_{k,0}}{f\left( k\right) +g\left(
k\right) -1}\,. 
\end{equation*}
\end{proof}

\begin{corollary}
In Theorem \ref{teo1}, for special choices of $a, b, f(k)$ and $g(k)$, the
following result can be obtained for well-known number sequences in
literature:
\begin{itemize}
\item If $f(k)=1,$ $g(k)=1,$ $a=0$ and $b=1$, for the Fibonacci sequence in \cite{Zhouveark}, we obtain $\left\Vert C_{n,g}\left( F\right) \right\Vert
_{2}=F_{n+2}-1$,
\item If $f(k)=1,$ $g(k)=1,$ $a=2$ and $b=1$, for the Lucas sequence in \cite{Zhouveark}, we obtain $\left\Vert C_{n,g}\left( L\right) \right\Vert
_{2}=L_{n+2}-3$,
\item If $f(k)=2,$ $g(k)=1,$ $a=0$ and $b=1$, for the Pell sequence, we
obtain $\left\Vert C_{n,g}\left( P\right) \right\Vert _{2}=\dfrac{
P_{n+1}+P_{n}-1}{2}$,
\item If $f(k)=1,$ $g(k)=2,$ $a=0$ and $b=1$, for the Jacobsthal sequence,
we obtain $\left\Vert C_{n,g}\left( J\right) \right\Vert _{2}=\dfrac{
J_{n+2}-1}{2}$,
\item Finally, we should note that choosing suitable values on $f(k),
g(k), a$ and $b$ in Theorem \ref{teo1}, it is actually obtained the spectral
norms of $g$-circulant matrix for the others second order sequences such as
$k$-Fibonacci, $k$-Lucas, Pell-Lucas, Jacobsthal-Lucas, Horadam, etc.
\end{itemize}
\end{corollary}

\begin{theorem}\label{teo2}
Let $C_{n,g}\left( H\right) =g$-$circ\left( H_{k,1},H_{k,2},\ldots
H_{k,n}\right)$ be circulant matrix as in (\ref{2.1}). Then we have
\begin{equation*}
\det C_{n,g}\left( H\right) =\det Q_{g}\cdot \left[
H_{k,1}N^{n-1}+H_{k,1}M^{n-2}\sum\limits_{i=1}^{n-1}\left( -\frac{
H_{k,2}H_{k,i+1}}{H_{k,1}}+H_{k,i+2}\right) \left( \frac{N}{M}\right) ^{i-1}%
\right]\, ,
\end{equation*}
where $M=g(k)(H_{k,n}-H_{k,0})$, $N=H_{k,1}-H_{k,n+1}$ and $\left(
n,g\right) =1$\,.
\end{theorem}

\begin{proof}
By using Lemma \ref{lem2} ve \ref{lem3}, we can write 
\begin{equation*}
C_{n,g}\left( H\right) =Q_{g}C_{n}\left( H\right)\,, 
\end{equation*}
where $\left( n,g\right) =1,\ Q_{g}$ is a $g$-circulant matrix and $
C_{n}\left( H\right)$ is a circulant matrix with generalized $k$-Horadam
number. From properties of determinant function and Equation (\ref{1.4}), the
proof is complete.
\end{proof}

\begin{corollary}
In Theorem \ref{teo2}, for special choices of $a, b, f(k)$ and $g(k)$, the
following result can be obtained for well-known number sequences in
literature:
\begin{itemize}
\item If $f(k)=1,$ $g(k)=1,$ $a=0$ and $b=1$, for the Fibonacci sequence, we
obtain $\det C_{n,g}\left( F\right) =\det Q_{g}\cdot \left[ \left(
1-F_{n+1}\right) ^{n-1}+F_{n}^{n-2}\sum\limits_{i=1}^{n-1}F_{i}\left( \frac{
1-F_{n+1}}{F_{n}}\right) ^{i-1}\right] $
\item If $f(k)=1,$ $g(k)=1,$ $a=2$ and $b=1$, for the Lucas sequence, we
obtain $\det C_{n,g}\left( L\right) =\det Q_{g}\cdot \left[ \left(
1-L_{n+1}\right) ^{n-1}+\left( L_{n}-2\right)
^{n-2}\sum\limits_{i=1}^{n-1}\left( L_{i+2}-3L_{i+1}\right) \left( \frac{
1-L_{n+1}}{L_{n}-2}\right) ^{i-1}\right] $
\item If $f(k)=2,$ $g(k)=1,$ $a=0$ and $b=1$, for the Pell sequence, we
obtain $\det C_{n,g}\left( P\right) =\det Q_{g}\cdot \left[ \left(
1-P_{n+1}\right) ^{n-1}+P_{n}^{n-2}\sum\limits_{i=1}^{n-1}P_{i}\left( \frac{
1-P_{n+1}}{P_{n}}\right) ^{i-1}\right] $
\item If $f(k)=1,$ $g(k)=2,$ $a=0$ and $b=1$, for the Jacobsthal sequence,
we obtain $\det C_{n,g}\left( J\right) =\det Q_{g}\cdot \left[ \left(
1-J_{n+1}\right)
^{n-1}+2^{n-1}J_{n}^{n-2}\sum\limits_{i=1}^{n-1}J_{i}\left( \frac{1-J_{n+1}
}{2J_{n}}\right) ^{i-1}\right] $
\item Finally, we should note that choosing suitable values on $f(k)$, $
g(k)$, $a$ and $b$ in Theorem \ref{teo2}, it is actually obtained the determinant
of $g$-circulant matrix for the others second order sequences such as
$k$-Fibonacci, $k$-Lucas, Pell-Lucas, Jacobsthal-Lucas, Horadam, etc.
\end{itemize}
\end{corollary}

\begin{proposition}\label{prop1}
Let $C_{n,g}\left( H\right) =g$-$circ\left( H_{k,1},H_{k,2},\ldots
H_{k,n}\right) $ be $g$-circulant matrix as in (\ref{2.1}). For $n>2,$ $
C_{n,g}\left( H\right)$ is an invertible matrix.
\end{proposition}

\begin{proof}
By using Lemma \ref{lem2} ve \ref{lem3}, we can write $C_{n,g}\left( H\right)
=Q_{g}C_{n}\left( H\right)$, where $\left( n,g\right) =1, Q_{g}$ is a $g$
-circulant matrix and $C_{n}\left( H\right)$ is a circulant matrix with
generalized $k$-Horadam number. From the Equation (\ref{1.5}), $C_{n}\left(
H\right) $ is invertible for $n>2.$ Hence, $C_{n,g}\left( H\right) $ is an
invertible matrix, since $C_{n}\left( H\right) $ and $Q_{g}$ are invertible.
\end{proof}

\begin{theorem}\label{teo3}
Let $C_{n,g}\left( H\right) =g$-$circ\left( H_{k,1},H_{k,2},\ldots
H_{k,n}\right) $ be $g$-circulant matrix as in (\ref{2.1}), for $n>2,$ then, we have
\begin{eqnarray*}
C_{n,g}^{-1}(H) &=&[circ(\frac{1+f(k)S_{n}^{(n-2)}+g(k)S_{n}^{(n-3)}}{h_{n}},
\frac{g(k)S_{n}^{(n-2)}-\frac{H_{k,2}}{H_{k,1}}}{h_{n}},-\frac{S_{n}^{(1)}}{
h_{n}}, \\
&&-\frac{S_{n}^{(2)}-f(k)S_{n}^{(1)}}{h_{n}},-\frac{
S_{n}^{(3)}-f(k)S_{n}^{(2)}-g(k)S_{n}^{(1)}}{h_{n}},\ldots , \\
&&-\frac{S_{n}^{(n-2)}-f(k)S_{n}^{(n-3)}-g(k)S_{n}^{(n-4)}}{h_{n}})]\cdot
Q_{g}^{T},
\end{eqnarray*}
where $S_{n}^{(j)}=\sum\limits_{i=1}^{j}\frac{\left( H_{k,j+3-i}-
\frac{H_{k,2}H_{k,j+2-i}}{H_{k,1}}\right) M^{i-1}}{N^{i}}$ $(j=1,2,\ldots
,n-2),\ h_{n}=-\dfrac{H_{k,2}H_{k,n}}{H_{k,1}}+H_{k,1}+\sum
\limits_{i=1}^{n-1}\left( -\frac{H_{k,2}H_{k,i+1}}{H_{k,1}}+H_{k,i+2}\right)
\left( \frac{M}{N}\right) ^{n-\left( i+1\right) },$ $M=g(k)(H_{k,n}-H_{k,0})$
and $N=H_{k,1}-H_{k,n+1}.$
\end{theorem}

\begin{proof}
The proofs of theorem can be done similarly by considering Proposition \ref{prop1}.
\end{proof}

\begin{corollary}
In Theorem \ref{teo3}, for special choices of $a,\ b,\ f(k)$ and $g(k)$, the
following result can be obtained for well-known number sequences in
literature:

\begin{itemize}
\item If $f(k)=1$, $g(k)=1$, $a=0$ and $b=1$, for the classic Fibonacci
sequence, we obtain
$C_{n,g}^{-1}(F)=\left[ 
\begin{array}{c}
\frac{1}{f_{n}}circ(1+\sum\limits_{i=1}^{n-2}\frac{F_{n-i}F_{n}^{i-1}}{
\left( F_{1}-F_{n+1}\right) ^{i}},-1+\sum\limits_{i=1}^{n-2}\frac{
F_{n-1-i}F_{n}^{i-1}}{\left( F_{1}-F_{n+1}\right) ^{i}},-\frac{1}{
F_{1}-F_{n+1}}, \\ 
-\frac{F_{n}}{(F_{1}-F_{n+1})^{2}},\linebreak -\frac{F_{n}^{2}}{%
(F_{1}-F_{n+1})^{3}},\ldots ,-\frac{F_{n}^{n-3}}{(F_{1}-F_{n+1})^{n-2}})
\end{array}
\right] \cdot Q_{g}^{T}$,
where $f_{n}=F_{1}-F_{n}+\sum\limits_{i=1}^{n-2}F_{i}\left( \frac{F_{n}}{
F_{1}-F_{n+1}}\right) ^{n-(i+1)}$.
\item If $f(k)=1$, $g(k)=1$, $a=2$ and $b=1$, for the classic Lucas
sequence, we obtain
$C_{n,g}^{-1}(L)=\left[ 
\begin{array}{c}
\frac{1}{l_{n}}circ(1+\sum\limits_{i=1}^{n-2}\frac{\left(
L_{n+2-i}-3L_{n+1-i}\right) \left( L_{n}-2\right) ^{i-1}}{\left(
L_{1}-L_{n+1}\right) ^{i}}, \\ 
-3+\sum\limits_{i=1}^{n-2}\frac{\left( L_{n+1-i}-3L_{n-i}\right) \left(
L_{n}-2\right) ^{i-1}}{\left( L_{1}-L_{n+1}\right) ^{i}},\linebreak \frac{5}{
L_{1}-L_{n+1}}, \\ 
\frac{5(L_{n}-2)}{(L_{1}-L_{n+1})^{2}},\frac{5(L_{n}-2)^{2}}{
(L_{1}-L_{n+1})^{3}},\ldots ,\frac{5(L_{n}-2)^{n-3}}{(L_{1}-L_{n+1})^{n-2}})
\end{array}
\right] \cdot Q_{g}^{T}$, \linebreak where $l_{n}=L_{1}-3L_{n}+\sum
\limits_{i=1}^{n-2}\left( L_{i+2}-3L_{i+1}\right) \left( \frac{L_{n}-2}{
L_{1}-L_{n+1}}\right) ^{n-(i+1)}$.
\item If $f(k)=2$, $g(k)=1$, $a=0$ and $b=1$, for the classic Pell sequence,
we obtain $C_{n,g}^{-1}(P)=\left[ 
\begin{array}{c}
\frac{1}{p_{n}}circ(1+\sum\limits_{i=1}^{n-2}\frac{P_{n-i}P_{n}^{i-1}}{
\left( P_{1}-P_{n+1}\right) ^{i}},-2+\sum\limits_{i=1}^{n-2}\frac{
P_{n-1-i}P_{n}^{i-1}}{\left( P_{1}-P_{n+1}\right) ^{i}}, \\ 
-\frac{1}{P_{1}-P_{n+1}},-\frac{P_{n}}{(P_{1}-P_{n+1})^{2}},\linebreak -
\frac{P_{n}^{2}}{(P_{1}-P_{n+1})^{3}},\ldots ,-\frac{P_{n}^{n-3}}{
(P_{1}-P_{n+1})^{n-2}})
\end{array}
\right] \cdot Q_{g}^{T}$,
where $p_{n}=P_{1}-2P_{n}+\sum\limits_{i=1}^{n-2}P_{i}\left( \frac{P_{n}}{
P_{1}-P_{n+1}}\right) ^{n-(i+1)}$.
\item If $f(k)=1$, $g(k)=2$, $a=0$ and $b=1$, for the classic Jacobsthal
sequence, we obtain $C_{n,g}^{-1}(J)=\left[ 
\begin{array}{c}
\frac{1}{s_{n}}circ(1+2\sum\limits_{i=1}^{n-2}\frac{J_{n-i}\left(
2J_{n}\right) ^{i-1}}{\left( J_{1}-J_{n+1}\right) ^{i}},-1+4\sum
\limits_{i=1}^{n-2}\frac{J_{n-1-i}\left( 2J_{n}\right) ^{i-1}}{\left(
J_{1}-J_{n+1}\right) ^{i}}, \\ 
-\frac{2}{J_{1}-J_{n+1}},-\frac{2^{2}J_{n}}{(J_{1}-J_{n+1})^{2}},\linebreak -
\frac{2^{3}J_{n}^{2}}{(J_{1}-J_{n+1})^{3}},\ldots ,-\frac{2^{n-2}J_{n}^{n-3}
}{(J_{1}-J_{n+1})^{n-2}})
\end{array}
\right] \cdot Q_{g}^{T}$,
where $s_{n}=J_{1}-J_{n}+2\sum\limits_{i=1}^{n-2}J_{i}\left( \frac{2J_{n}}{
J_{1}-J_{n+1}}\right) ^{n-(i+1)}$.
\item If $f(k)=p~$and $g(k)=q$, for the classic Horadam sequence, we obtain
$C_{n,g}^{-1}(W)=\left[ 
\begin{array}{c}
\frac{1}{z_{n}}circ(1+\sum\limits_{i=1}^{n-2}\left( w_{n+2-i}-\frac{w_{2}}{
w_{1}}w_{n+1-i}\right) \frac{A^{i-1}}{B^{i}}, \\ 
\frac{-w_{2}}{w_{1}}+q\sum\limits_{i=1}^{n-2}\frac{\left( w_{n+1-i}-\frac{
w_{2}w_{n-i}}{w_{1}}\right) A^{i-1}}{B^{i}},\linebreak -\frac{w_{3}-\frac{
w_{2}^{2}}{w_{1}}}{B}, \\ 
-\frac{\left( w_{3}-\frac{w_{2}^{2}}{w_{1}}\right) A}{B^{2}},-\frac{\left(
w_{3}-\frac{w_{2}^{2}}{w_{1}}\right) A^{2}}{B^{3}},\ldots ,-\frac{\left(
w_{3}-\frac{w_{2}^{2}}{w_{1}}\right) A^{n-3}}{B^{n-2}})
\end{array}
\right] \cdot Q_{g}^{T}$,
where $z_{n}=-\frac{w_{2}w_{n}}{w_{1}}+w_{1}+\sum\limits_{i=1}^{n-2}\left( -
\frac{w_{2}w_{i+1}}{w_{1}}+w_{i+2}\right) \left( \frac{A}{B}\right)
^{n-(i+1)},~A=q(w_{n-}w_{0})$ and$~B=w_{1}-w_{n+1}$.
\item Finally, we should note that choosing suitable values on $f(k)$, $
g(k)$, $a$ and $b$ in Theorem \ref{teo3}, it is actually obtained inverse of $g$
-circulant matrix for the others second order sequences such as $k$-Fibonacci, $k$-Lucas, Pell-Lucas, Jacobsthal-Lucas, Horadam sequences.
\end{itemize}
\end{corollary}

\begin{conclusion}
In this paper, we introduced the $g$-circulant matrix with the generalized $
k$- Horadam numbers and presented some properties of this matrix. By the
results in Sections 2 of this paper, we have a great opportunity to obtain
norm, determinant and inverse of the circulant matrices with second order
number sequences. Thus, we extend some recent result in the literature.
In the future studies on the circulant matrix for number sequences, we
except that the following topics will bring a new insight. For example, it
would be interesting to study the $g$-circulant matrix for third order
number sequences.
\end{conclusion}

\end{document}